\documentclass[11pt,reqno]{amsart}
\usepackage[utf8]{inputenc}
\usepackage[T1]{fontenc}
\usepackage[english]{babel}
\usepackage{latexsym,amsmath,amsfonts,amscd,amssymb}
\usepackage{graphics}
\usepackage{float}
\usepackage[all]{xy}
\usepackage{soul}
\usepackage{amsthm}
\usepackage{booktabs}
\usepackage{tabu}
\usepackage[font={small,it}]{caption}
\usepackage{color}
\definecolor{cobalt}{RGB}{61,89,171}
\usepackage[colorlinks,citecolor=cobalt,linkcolor=cobalt,urlcolor=cobalt,pdfpagemode=UseNone,backref = page]{hyperref}
\usepackage{makecell}

\usepackage{adjustbox}
\usepackage{longtable}
\usepackage{array, multicol, multirow, makecell}
\usepackage{enumerate}
\usepackage{enumitem}
\usepackage{lscape}

\DeclareMathOperator{\Id}{Id}

\usepackage{mathtools}

\newcommand{\fX}{{\mathfrak X}}

\newcommand{\CC}{{\mathbb C}}

\newcommand{\RR}{{\mathbb R}}

\newcommand{\frd}{{\mathfrak{d}}}
\newcommand{\frg}{{\mathfrak{g}}}
\newcommand{\frh}{{\mathfrak{h}}}

\newcommand{\fq}{{\mathfrak{q}}}

\newtheorem{theorem}{Theorem}[section]
\newtheorem{proposition}[theorem]{Proposition}
\newtheorem{lemma}[theorem]{Lemma}
\newtheorem{corollary}[theorem]{Corollary}

\theoremstyle{definition}
\newtheorem{definition}[theorem]{Definition}
\newtheorem{example}[theorem]{Example}

\theoremstyle{remark}

\newtheorem{remark}[theorem]{Remark}

\title{Symmetric and skew-symmetric complex structures}

\author[G. Bazzoni]{Giovanni Bazzoni}
\address{Dipartimento di Scienza ed Alta Tecnologia, Università degli Studi dell'Insubria, Via Valleggio 11, 22100, Como, Italy}
\email{giovanni.bazzoni@uninsubria.it}

\author[A. Gil-García]{Alejandro Gil-García}
\address{Universidad Complutense de Madrid, Madrid, Spain}
\email{alegil05@ucm.es}

\author[A. Latorre]{Adela Latorre}
\address{Departamento de Matem\'atica Aplicada, Universidad Polit\'ecnica de Madrid, C/ Jos\'e Antonio Novais 10, 28040 Madrid, Spain}
\email{adela.latorre@upm.es}


\begin{document}

\begin{abstract}
On a complex manifold $(M,J)$, we interpret complex symplectic and pseudo-Kähler structures as symplectic forms with respect to which $J$ is, respectively, symmetric and skew-symmetric. We classify complex symplectic structures on 4-dimensional Lie algebras. We develop a method for constructing hypersymplectic structures from the above data. This allows us to obtain an example of a hypersymplectic structure on a 4-step nilmanifold.
\end{abstract}

\maketitle


\section{Introduction}\label{sec:1}

It is customary to say that Kähler geometry lies in the intersection of complex, symplectic, and Riemannian geometry. In fact, a Kähler structure on a manifold $M$ can be thought of as a pair $(J,\omega)$, where $J$ is a complex structure and $\omega$ is a symplectic form, such that, for vector fields $X,Y\in\fX(M)$, $g(X,Y)=\omega(X,JY)$ defines a Riemannian metric on $M$. This requires both the tameness and the compatibility of $J$ with $\omega$, that is $\omega(X,JX)>0$ and $\omega(JX,JY)=\omega(X,Y)$. The latter condition ensures that $g$ is symmetric, the former that it is positive-definite. In fact, for 4-manifolds, Donaldson asks whether tameness implies compatibility (see \cite[Question 2]{Donaldson}). Dropping the tameness assumption, one enters the realm of pseudo-Kähler structures. The compatibility condition can be rewritten as $\omega(JX,Y)=-\omega(X,JY)$: with respect to $\omega$, $J$ is a {\em skew-symmetric} endomorphism. This is equivalent to $\omega\in\Omega^{1,1}(M)$.

In the same paper, Donaldson considers various geometric structures on a 4-manifold, which are described in terms of certain two-forms and a complex structure: symplectic, complex symplectic, Kähler, hyperKähler\ldots Clearly, all these structures can be defined in dimensions other than 4. For instance, a complex symplectic structure on a complex manifold is a holomorphic two-form which is closed and non degenerate; in particular, the complex manifold has even complex dimension and its canonical bundle is trivial. It is known that a complex symplectic structure can be equivalently given in terms of a complex structure $J$ and a real symplectic form $\sigma$, subject to a different compatibility condition, namely $\sigma(JX,Y)=\sigma(X,JY)$, for $X,Y\in\fX(M)$: with respect to $\sigma$, $J$ is a {\em symmetric} endomorphism. This is also equivalent to $\sigma\in\Omega^{2,0}(M)\oplus\Omega^{0,2}(M)$.

The aim of this paper is to study, on a fixed complex manifold $(M^{4n},J)$, the geometry of two symplectic forms $\sigma$ and $\omega$ with respect to which $J$ is, respectively, symmetric and skew-symmetric. Both hyperKähler \cite{Hitchin1} and hypersymplectic \cite{DancerSwann,Hitchin2} geometry can arise in this way. In the hyperKähler case, the pseudo-Kähler metric $g$ is actually Kähler, in fact $\mathrm{Hol}(g)\subset\mathrm{Sp}(n)$, while it is neutral in the hypersymplectic case, and $\mathrm{Hol}(g)\subset\mathrm{Sp}(2n,\mathbb{R})$. In both cases, $g$ is Ricci-flat. A related approach to these geometries has been taken in~\cite{Bande-Kotschick}. The importance of hyperKähler metrics in Physics, especially in supersymmetry, is well-known \cite{Hitchin3}. Hypersymplectic metrics as well are relevant in Physics, notably in $N=2$ string theory \cite{OoguriVafa}.

The kind of manifolds we are interested in are compact nilmanifolds and, more generally, solvmanifolds. These are quotients of a connected, simply connected, nilpotent (resp.~solvable) Lie group by a lattice, i.e.~a discrete and co-compact subgroup. On the one hand, it is well-known that a Kähler nilmanifold is diffeomorphic to a torus \cite{Hasegawa} and that a Kähler solvmanifold is a finite quotient of a complex torus and a complex torus bundle over a complex torus \cite{Hasegawa2}. On the other hand, over the last years, a great deal of research has been devoted to pseudo-Kähler and complex symplectic structures on nilpotent and solvable Lie algebras, see for instance \cite{BFLM2018,CPO,ContiRossi,CFU04,LatorreUgarte}. Hypersymplectic structures on nilpotent and solvable Lie algebras have been studied in \cite{Andrada2006,AndradaDotti,Guan11,NiBai}. Such structures at the Lie algebra level provide left-invariant structures on the corresponding connected, simply connected Lie groups, and on compact quotients thereof. For this reason, we will mainly focus on the study of geometric structures on Lie algebras.

This paper is organized as follows:
\begin{itemize}
\item in Section \ref{sec:Preliminaries} we recall the basic notions on geometric structures on Lie algebras;
\item in Section \ref{sec:CS4D} we consider 4-dimensional Lie algebras endowed with a complex structure and classify complex symplectic structures, that is, symplectic forms with respect to which the complex structure is symmetric (see Theorem \ref{theo:complex-symplectic_classification});
\item in Section \ref{section:hypersymplectic} we show how hypersymplectic structures, on a Lie algebra endowed with a fixed complex structure, can be constructed from the auxiliary data of two symplectic structures, such that the complex structure is symmetric with respect to one, and skew-symmetric with respect to the other one (see Theorem \ref{theo:main}); we also provide an example in which the symplectic forms can never be combined to obtain a hypersymplectic structure;
\item in Section \ref{sec:examples} we use Theorem \ref{theo:main} to construct two 1-parameter families of left-invariant hypersymplectic structures on a nilmanifold $\Gamma\backslash G$ with $\frg=\mathrm{Lie}(G)$ 4-step nilpotent; to the best of our knowledge, this is the first example with such nilpotency step. Both metrics are complete, one is flat and the other one is non-flat (see Theorem \ref{theo:example}). 
\end{itemize}

\noindent\textbf{Acknowledgements.} The first author is supported by GNSAGA of INdAM. The third author has been partially supported by the projects MTM2017-85649-P (AEI/FEDER, UE), and E22-17R ``\'Algebra y Geometr\'ia'' (Gobierno de Arag\'on/FEDER).


\section{Preliminaries}\label{sec:Preliminaries}

We briefly recall some basic notions about Lie algebras admitting different types of geometric structures. We will also fix the notation that will be used throughout the paper.


Let $\frg$ be an $n$-dimensional Lie algebra. It is well-known that
$\Lambda^*\frg^*=\bigoplus_{r\geq 0}\Lambda^r\frg^*$ is a differential graded algebra with $d\colon\Lambda^r\frg^*\to\Lambda^{r+1}\frg^*$; in fact, $d\colon\frg^*\to\Lambda^2\frg^*$ is the dual of the Lie bracket:
$$d\alpha(X,Y)=-\alpha\big([X,Y] \big), \ \ \forall\alpha\in\frg^*, \ \forall X,Y\in\frg.$$
Consequently, if we fix a basis $\{e^k\}_{k=1}^n$ of $\frg^*$, then $\frg$ is fully determined by the expressions:
\[
de^k=\sum_{1\leq i<j\leq n} c^k_{ij}\,e^i\wedge e^j, \quad 1\leq k\leq n\,,
\]
which are known as the \emph{structure equations}. In the rest of the paper, we will denote $e^{ij}:=e^i\wedge e^j$. Moreover, we will make use of Salamon's notation: for instance, if a $4$-dimensional Lie algebra $\frg$ is defined by structure equations $de^1=2\,e^{14}$, $de^2=-e^{24}$, $de^3=-e^{12}+e^{34}$, $de^4=0$, we will write $\frg=(2\cdot 14,-24,-12+34,0)$.


An endomorphism $A:\frg\longrightarrow\frg$ is called \emph{integrable} if the Nijenhuis tensor
\begin{equation}\label{nijenhuis}
N_A(X,Y):=-A^2[X,Y]+A[AX,Y]+A[X,AY]-[AX,AY]
\end{equation}
vanishes identically. 

\subsection{Complex structures on Lie algebras}\label{subsec:complexOnLie}

Let us suppose that the dimension of the Lie algebra $\frg$ is even, namely, $n=2m$.
An \emph{almost complex structure} on~$\frg$ is an endomorphism $J:\frg\to\frg$ satisfying $J^2=-\mathrm{Id}$. Naturally extending $J$ to the complexification $\frg_{\mathbb C}=\frg\otimes\mathbb C$, one observes that $J$ has eigenvalues~$i$ and~$-i$ with eigenspaces respectively given by
$$\frg^{(1,0)}=\{X-i\,JX\mid X\in\frg\}, \qquad 
\frg^{(0,1)}=\{X+i\,JX\mid X\in\frg\}.$$
Note that $J$ can be equivalently defined on $\frg^*$ using the expression $J\alpha=\alpha\circ J$, for any $\alpha\in\frg^*$, which provides a similar splitting on $\frg^*_\mathbb C=\frg^*\otimes\mathbb C$:
$$\frg^{*(1,0)}=\{\alpha-i\,J\alpha\mid \alpha\in\frg^*\}, \qquad 
\frg^{*(0,1)}=\{\alpha+i\,J\alpha\mid \alpha\in\frg^*\}.$$
This splitting extends to the space $\Lambda^*\frg^*_{\mathbb C}$, in such a way that for every $r>0$ one has
\[
\Lambda^r\frg^*_{\mathbb C}=\bigoplus_{p+q=r}\Lambda^p\frg^{*(1,0)}\otimes\Lambda^q\frg^{*(0,1)}\eqqcolon\bigoplus_{p+q=r}\Lambda^{p,q}\frg^*_\CC\,.
\]

If the almost complex structure $J$ additionally satisfies the integrability condition $N_J\equiv 0$, where $N_J$ is defined by~\eqref{nijenhuis}, then~$J$ is called a \emph{complex structure}. In this case,~$\frg^{(1,0)}$ and~$\frg^{(0,1)}$ are Lie subalgebras of~$\frg_{\mathbb C}$ and 
\[
d\left( \frg^{*(1,0)}\right)\subseteq \Lambda^{2,0}\frg^*_\CC\oplus\Lambda^{1,1}\frg^*_\CC\,;
\]
in fact, these are equivalent conditions for the integrability of $J$.

By a result of Samelson \cite{Samelson1953}, every even-dimensional compact Lie group admits at least one left-invariant complex structure. However, it is important to observe that not every Lie algebra admits a complex structure. Indeed, the existence of complex structures on (arbitrary) nilpotent or solvable Lie algebras is nowadays an open problem and an active field of research (see~\cite{ABD, COUV, GR, Guan11, Mil, Ovando2000, Salamon2001}, among others).

\subsection{Symplectic forms and (skew-)symmetric complex structures}\label{subsec:skew}
A \emph{symplectic structure} on a $2m$-dimensional Lie algebra $\frg$ is a closed $\omega\in\Lambda^2\frg^*$ such that $\omega^m\neq 0$. If~$\frg$ additionally admits a complex structure $J$, it is natural to wonder whether there exists some kind of relation between $J$ and $\omega$. In this sense, one can introduce the following terminology:

\begin{definition}
Let $\frg$ be a Lie algebra endowed with a complex structure $J$ and a symplectic form $\omega$. The complex structure $J$ is said to be
\begin{itemize}
\item \emph{symmetric} with respect to $\omega$ if $\omega(JX,Y)=\omega(X,JY)$, for every $X,Y\in\frg$;
\item \emph{skew-symmetric} with respect to $\omega$ if $\omega(JX,Y)=-\omega(X,JY)$, for every $X,Y\in\frg$.
\end{itemize}
\end{definition}

Each of the above conditions is related to a different type of geometric structure.

In the case of symmetric complex structures, it has been shown in~\cite[Lemma 3.2]{BFLM2018} that they are in one-to-one correspondence with complex symplectic structures. Recall that a \emph{complex symplectic structure} on a $2m$-dimensional real Lie algebra $\frg$ is a pair $(J,\omega_{\mathbb C})$ consisting of a complex structure~$J$ and an element $\omega_{\mathbb C}\in\Lambda^{2,0}\frg^{*}_{\mathbb C}$ which is closed and non-degenerate. The existence of the holomorphic form~$\omega_{\mathbb C}$ requires the complex dimension~$m$ of $(\frg,J)$ to be even.

A skew-symmetric complex structure naturally gives rise to a pseudo-Hermitian metric $g(X,Y)=\omega(JX,Y)$, for $X,Y\in\frg$, whose fundamental form is, precisely, $\omega\in\Lambda^{1,1}\frg^*_\CC$. Since~$\omega$ is closed, the metric~$g$ is known as \emph{pseudo-K\"ahler}. In the particular case when~$g$ is positive definite, the triple $(J,\omega,g)$ is called a K\"ahler structure.

There is in general no relation between the existence of pseudo-K\"ahler and complex symplectic structures, as shown by Yamada~\cite{Yamada2017}. However, in this paper we are particularly interested in their coexistence, especially when they share the same complex structure.


\subsection{Hypersymplectic structures on Lie algebras}\label{subsec:hyperSympl}
We shall see that a hypersymplectic structure on a Lie algebra is, in particular, an example of a complex structure which is symmetric with respect to some symplectic form and skew-symmetric with respect to some other symplectic form; in other words, hypersymplectic geometry is a special case of both complex symplectic and pseudo-K\"ahler geometry.

An {\em almost product structure} on a Lie algebra $\frg$ is an endomorphism $E\colon\frg\to\frg$ satisfying $E^2=\Id$, and not equal to $\pm\Id$. A {\em product structure} is an almost product structure which is integrable, that is, $N_E\equiv0$. Given an almost product structure $E$ on $\frg$, we can write $\frg=\frg_+\oplus\frg_-$, where $\frg_{\pm}$ is the eigenspace with eigenvalue $\pm1$ of $E$. It is easy to see that the integrability of $E$ is equivalent to $\frg_{\pm}$ being Lie subalgebras of $\frg$.

By combining a product structure with a complex structure one gets the notion of {\em complex product structure} on a Lie algebra $\frg$: this is a pair $(J,E)$ formed by a complex structure $J$ and a product structure $E$ which anti-commute: $J\circ E=-E\circ J$. In this case, $J\circ E$ is also a product structure on $\frg$ and $J\colon\frg_\pm\to\frg_\mp$ is an isomorphism; in particular, $\dim\frg_+=\dim\frg_-$. Notice that the dimension of a Lie algebra with a complex product structure is even but not necessarily a multiple of 4.
	
Let $\frg$ be a Lie algebra endowed with a complex product structure $(J,E)$ and let $g$ be a scalar product on $\frg$. We say that $g$ is {\em compatible} with $(J,E)$ if 
\begin{equation}\label{eq:compatible_metric}
g(JX,JY)=g(X,Y)\quad\text{and}\quad g(EX,EY)=-g(X,Y)\,, \quad  \forall X,Y\in\frg\,.
\end{equation}	
The compatibility condition implies that $\frg_\pm^\bot=\frg_\pm$, hence $g$ has signature $(m,m)$, where $2m=\dim\frg$. For $X,Y\in\frg$, define the following bilinear forms on $\frg$: 
\begin{equation}\label{eq:bilinear_forms}
\omega_1(X,Y)=g(JX,Y)\,, \quad \omega_2(X,Y)=g(EX,Y)\,, \quad \ \omega_3(X,Y)=g\big((J\circ E)X,Y\big)\,.
\end{equation}
Note that $\omega_2(X,Y)=\omega_3(JX,Y)$, hence one obtains $\omega_2$ from $\omega_3$ and $J$, and viceversa. Thanks to compatibility \eqref{eq:compatible_metric}, $\omega_i\in\Lambda^2\frg^*$ for $i=1,2,3$; moreover, they are non-degenerate, since $g$ is non-degenerate and $J$ and $E$ are isomorphisms. Let $\omega_\pm$ denote the restriction of $\omega_1$ to $\frg_\pm$. It follows from \cite[Lemma 3]{Andrada2006} that both $\omega_+$ and $\omega_-$ are non-degenerate. Hence $m=\dim\frg_+=\dim\frg_-$ must be even: $m=2n$. Therefore, $\dim\frg=4n$ and the signature of~$g$ is $(2n,2n)$. 

We now turn to the case in which the 2-forms defined in \eqref{eq:bilinear_forms} are closed, hence symplectic. It turns out that if one of the 2-forms $\omega_1$ or $\omega_3$ is closed, then they are all closed.
	
\begin{proposition}{\normalfont\cite[Proposition 5]{Andrada2006}}\label{prop:all_closed}
Let $(J,E)$ be a complex product structure on $\frg$. Let~$\omega_i$, $i=1,2,3$, be the 2-forms on $\frg$ given by \eqref{eq:bilinear_forms}. Then the following are equivalent:
\begin{itemize}
\item $\omega_1$ is closed.
\item $\omega_3$ is closed.
\end{itemize}
Furthermore, if one of the conditions above holds, then $\omega_2$ is also closed.
\end{proposition}
	
\begin{definition}
Let $(J,E)$ be a complex product structure on the Lie algebra $\frg$ and let $g$ be a compatible metric. If $\omega_1$ or $\omega_3$ defined in \eqref{eq:bilinear_forms} are closed, we call $(J,E,g)$ a {\em hypersymplectic structure}; we refer to $\frg$ as a {\em hypersymplectic Lie algebra} and to $g$ as a {\em hypersymplectic metric}.
\end{definition}

This definition, together with \eqref{eq:bilinear_forms}, shows that if $(J,E,g)$ is a hypersymplectic structure on~$\frg$, then~$J$ is symmetric with respect to~$\omega_3$ and skew-symmetric with respect to~$\omega_1$; in other words, $(J,\omega_3)$ is a complex symplectic structure and $(J,\omega_1)$ is a pseudo-Kähler structure on $\frg$. In particular, $g$ is the pseudo-K\"ahler metric associated with $(J,\omega_1)$.

From now on, we consider interchangeably $g$, $\omega_1$ and $\omega_3$ as tensors of type $(0,2)$ and as maps from $\frg$ to $\frg^*$; as an example, for $X\in\frg$, $g(X)\in\frg^*$ is given by $g(X)(Y)=g(X,Y)$, for every $Y\in\frg$. The non-degeneracy ensures that all such maps are isomorphisms. 

The following lemma provides a key observation that motivates the construction developed in Section~\ref{section:hypersymplectic}.

\begin{lemma}\label{prop:E_hypersymplectic}
Let $(J,E,g)$ be a hypersymplectic structure on a Lie algebra $\frg$, with~$\omega_1$ and~$\omega_3$ defined as in \eqref{eq:bilinear_forms}. Then
\begin{equation}\label{eq:E_product_structure}
E=\omega_1^{-1}\circ\omega_3\,.
\end{equation}
\end{lemma}
\begin{proof}
By \eqref{eq:bilinear_forms} we have
$\omega_1=g\circ J$ and $\omega_3=g\circ J\circ E=\omega_1\circ E$, hence $E=\omega_1^{-1}\circ\omega_3$.
\end{proof}

\begin{remark}\label{rem:recursion}
Equation \eqref{eq:E_product_structure} can be rewritten as $\omega_1(EX)=\omega_3(X)$, hence $E$ is an example of a {\em recursion operator}, a terminology coming from Physics (see for instance \cite{Gutkin85,Zakharov84}). The study of the geometry of recursion operators has been undertaken by Bande and Kotschick in \cite{Bande-Kotschick}; in fact, there, a hypersymplectic structure is understood as a triple of symplectic structures whose recursion operators satisfy certain conditions.
\end{remark}


\section{Symmetric and skew-symmetric complex structures in 4d}\label{sec:CS4D}

In this section we classify 4-dimensional Lie algebras that admit a complex structure $J$ and a symplectic form $\omega$ such that $J$ is symmetric with respect to $\omega$; in other words, we classify 4-dimensional complex symplectic Lie algebras.

Let us first recall that a 4-dimensional symplectic Lie algebra is necessarily solvable, by a result of Chu \cite{Chu1974}. One of the first classifications of 4-dimensional solvable Lie algebras is given in \cite{Mubarakzyanov1963}. Among them, those admitting a symplectic structure can be found in~\cite{Ovando2006Sym}.
Complex structures on 4-dimensional solvable Lie algebras were classified in \cite{Snow1990} and \cite{Ovando2000}. 

Since we are interested in 4-dimensional solvable Lie algebras $\frg$ simultaneously admitting symplectic and complex structures, we have crossed the previous results and summarized them in Table~\ref{tabla1}. For the notation of Lie algebras and the description of complex structures we follow~\cite{Ovando2004}. However, we here rewrite the complex structure recalling that the 2-dimensional complex subspace $\fq$ of $\frg\otimes\CC$ given in~\cite{Ovando2004} plays the role of $\frg^{(0,1)}$ (see Section~\ref{subsec:complexOnLie}). 
For instance, a complex structure on the Lie algebra $\mathfrak{r}'_2$ is defined by $\fq_\xi=\langle e_1+\xi e_2,e_3+ie_4\rangle$, $\xi\in\CC$, $\Im(\xi)\neq0$, which means $J_\xi(e_1+\xi e_2)=-i(e_1+\xi e_2)$ and $J_\xi(e_3+i e_4)=-i(e_3+i e_4)$. Taking the real and imaginary part, we see that the complex structure is then given by
\begin{equation}\label{Jxi}
J_\xi(e_1)=\frac{\Re(\xi)}{\Im(\xi)}e_1+\frac{|\xi|^2}{\Im(\xi)}e_2\quad \mathrm{and} \quad J_\xi(e_3)=e_4\,.
\end{equation}
Also notice that the elements appearing in the last column of Table~\ref{tabla1} are generic 2-cocycles, and that the additional condition ensures the non-degeneracy. We will work with this description of symplectic structures. Nonetheless, there is a notion of equivalence for symplectic Lie algebras; the classification of 4-dimensional symplectic real Lie algebras up to equivalence can be found in \cite{Ovando2006Sym}.

\begin{landscape}
\topskip2pt
\vspace*{\fill}

\renewcommand{\arraystretch}{1.4}
\begin{table}[htb!]
\scalebox{0.8}{
\begin{tabular}{|c| l | c | l | l |}
\hline
SLA & Structure equations & \multicolumn{2}{l |}{Complex structures} & Symplectic structures \\
\hline\hline
$\mathfrak{rh}_3$ & $(0,0,-12,0)$
	& \multicolumn{2}{l |}{$Je_1=-e_2, \ Je_3=-e_4$}
	& $\begin{array}{l}
		\omega=a_{12}e^{12}+a_{13}e^{13}+a_{14}e^{14}+a_{23}e^{23}+a_{24}e^{24},\\[-4pt]
		\text{with } a_{14}a_{23}-a_{13}a_{24}\neq 0
	\end{array}$\\ \hline\hline
$\mathfrak{rr}_{3,0}$ & $(0,-12,0,0)$ 
	& \multicolumn{2}{l |}{$Je_1=e_2, \ Je_3=e_4$}
	& $\begin{array}{l}
		\omega=a_{12}e^{12}+a_{13}e^{13}+a_{14}e^{14}+a_{34}e^{34},\\[-4pt]
		\text{with } a_{12}a_{34}\neq 0
	\end{array}$
	\\ \hline\hline
$\mathfrak{rr'}_{3,0}$ & $(0,-13,12,0)$
	& \multicolumn{2}{l |}{$Je_1=e_4, \ Je_2=e_3$}
	& $\begin{array}{l}
		\omega=a_{12}e^{12}+a_{13}e^{13}+a_{14}e^{14}+a_{23}e^{23},\\[-4pt]
		\text{with } a_{14}a_{23}\neq 0
	\end{array}$
	\\ \hline\hline
$\mathfrak{r}_2\mathfrak{r}_2$ & $(0,-12,0,-34)$
	& \multicolumn{2}{l |}{$Je_1=e_2, \ Je_3=e_4$}
	& $\begin{array}{l}
		\omega=a_{12}e^{12}+a_{13}e^{13}+a_{34}e^{34},\\[-4pt] 
		\text{with }a_{12}a_{34}\neq 0
	\end{array}$
	\\ \hline\hline
\multirow{3}{*}{$\mathfrak{r}'_2$} & \multirow{3}{*}{$(0,0,-13+24,-14-23)$} 
	& \multicolumn{2}{l |}{$Je_1=e_3, \ Je_2=e_4$}
	& \multirow{3}{*}{ $\begin{array}{l}
		\omega=a_{12}e^{12}+a_{13}(e^{13}-e^{24})+a_{14}(e^{14}+e^{23}),\\[-4pt]
		\text{with }a^2_{14}+a^2_{13}\neq 0
	\end{array}$ }
	\\ \cline{3-4}
	&	 
	& \multicolumn{2}{l |}{
		$J_{\xi}e_1=\frac{1}{\Im(\xi)}\left( \Re(\xi)\,e_1+|\xi|^2\,e_2\right), \ J_{\xi} e_3=e_4$,}
	& \\ 
	&  
	& \multicolumn{2}{l |}{where $\xi\in\mathbb C$, $\Im(\xi)\neq 0$}
	& \\ \hline\hline
$\mathfrak{r}_{4,-1,-1}$ & $(14,-24,-34,0)$ 
	& \multicolumn{2}{l |}{$Je_1=-e_4, \ Je_2=e_3$}
	& $\begin{array}{l}
		\omega=a_{12}e^{12}+a_{13}e^{13}+a_{14}e^{14}+a_{24}e^{24}+a_{34}e^{34},\\[-4pt] 
		\text{with }a_{12}a_{34}-a_{13}a_{24}\neq 0
	\end{array}$
	\\ \hline\hline
$\begin{array}{c} \mathfrak{r'}_{4,0,\delta} \\ \delta>0\end{array}$ 
	& $(14,\delta\,34,-\delta\,24,0)$ 
	& \multicolumn{2}{l |}{$J_{\pm}e_1=-e_4, \ J_{\pm}e_2=\pm e_3$}
	& $\begin{array}{l} 
		\omega=a_{14}e^{14}+a_{23}e^{23}+a_{24}e^{24}+a_{34}e^{34},\\[-4pt] 
		\text{with }a_{14}a_{23}\neq 0
	\end{array}$
	\\ \hline\hline
\multirow{5}{*}{$\begin{array}{c} \mathfrak{d}_{4,\lambda} \\ \lambda\geq \frac12\end{array}$} 
	& \multirow{5}{*}{$(\lambda\,14,(1-\lambda)\,24,-12+34,0)$} 
	& \multirow{2}{*}{$\lambda=\frac 12$}
	& $J_{\pm}e_1=\pm e_2, \ J_{\pm}e_3=-e_4$
	& \multirow{5}{*}{\!\!\!
	\begin{tabular}{l}$\omega=a_{12}(e^{12}-e^{34})+a_{14}e^{14}+a_{23}e^{23}+a_{24}e^{24}$,\\
	with $(\lambda-2)\,a_{23}=0$, $a^2_{12}-a_{14}a_{23}\neq 0$ \end{tabular}}
	\\ \cline{4-4}
 
	&  
	& 
	& $Je_1= -e_4, \ Je_2=-2\,e_3$
	& 	
	\\ \cline{3-4}
	&	
	& $\lambda=1$
	& $Je_1= e_4, \ Je_2=e_3$
	& 
	\\ \cline{3-4}
	&	
	& \multirow{2}{*}{$\lambda\neq \frac 12,1$}
	& $J_1e_1= \frac{1}{\lambda}\,e_4, \ J_1e_2=e_3$
	&  
	\\ \cline{4-4}
	&
	&
	& $J_2e_1= e_3, \ J_2e_2=\frac{1}{\lambda-1}\,e_4$
	& 
	\\ \hline\hline
\multirow{2}{*}{$\begin{array}{c}\mathfrak{d}'_{4,\delta}\\ \delta>0 \end{array}$} & 
	\multirow{2}{*}{$\big(\frac{\delta}{2}\,14+24,-14+\frac{\delta}{2}\,24,-12+\delta\,34,0\big)$} 
	& \multicolumn{2}{l |}{$Je_1=a\,e_2, \ Je_3=b\,e_4$}
	& \multirow{2}{*}{ $\begin{array}{l}
		\omega=a_{12}(e^{12}-\delta\,e^{34})+a_{14}e^{14}+a_{24}e^{24},\\[-4pt]
		\text{with }a_{12}\neq 0
	\end{array}$ }
	\\ 
	&	
	& \multicolumn{2}{l |}{where $a,\,b\in\{-1, 1\}$}
	& \\ \hline\hline
$\mathfrak{h}_{4}$ & $\big(\frac{1}{2}\,14+24,\frac{1}{2}\,24,-12+34,0\big)$ 
	& \multicolumn{2}{l |}{$Je_1=2\,e_3, \ Je_2= -e_4$}
	& $\begin{array}{l}
		\omega=a_{12}(e^{12}-e^{34})+a_{14}e^{14}+a_{24}e^{24},\\[-4pt] 
		\text{with }a_{12}\neq 0
	\end{array}$
	\\ \hline
\end{tabular}
}
\vskip 0.25 cm
\caption{$4$-dimensional solvable non-abelian Lie algebras admitting complex and symplectic structures}
\label{tabla1}
\end{table}
\vspace*{\fill}
\end{landscape}

\subsection{4-dimensional Lie algebras with a symmetric complex structure}\label{sec:4DSYM}

We classify 4-dimensional Lie algebras endowed with a complex structure which is symmetric with respect to a symplectic structure, that is, we classify complex symplectic 4-dimensional Lie algebras. For this classification, two complex symplectic Lie algebras $(\frg_1,J_1,\omega_1)$ and $(\frg_2,J_2,\omega_2)$ are considered {\em equivalent} if there exists a Lie algebra isomorphism $\varphi\colon \frg_1\to\frg_2$ such that $J_2\circ\varphi=\varphi\circ J_1$ and $\varphi^*\omega_2=\omega_1$.

\begin{theorem}\label{theo:complex-symplectic_classification}
A 4-dimensional Lie algebra $\frg$ admits a complex symplectic structure $(J,\omega)$ if and only if $(\frg,J,\omega)$ is one of the following:
\end{theorem}

\begin{table}[h!]
\begin{center}
{\tabulinesep=1.2mm
\begin{tabu}{lccccccc}
\toprule[1.5pt]
$\frg$ & Complex structure & Symplectic structure\\
\specialrule{1pt}{0pt}{0pt}
$\RR^4$ & $J(e_1)=-e_2$, $J(e_3)=-e_4$ & $\omega=e^{14}+e^{23}$\\
\specialrule{1pt}{0pt}{0pt}
$\mathfrak{rh}_3$ & $J(e_1)=-e_2$, $J(e_3)=-e_4$ & $\omega=e^{14}+e^{23}$\\
\specialrule{1pt}{0pt}{0pt}
$\mathfrak{r}'_2$ & $J_i(e_1)=e_2$, $J_i(e_3)=e_4$ & $\omega=e^{13}-e^{24}$\\
\specialrule{1pt}{0pt}{0pt}
$\mathfrak{r}_{4,-1,-1}$ & $J(e_1)=-e_4$, $J(e_2)=e_3$ & $\omega=e^{12}-e^{34}$\\
\specialrule{1pt}{0pt}{0pt}
$\mathfrak{d}_{4,2}$& $J_2(e_1)=e_3$, $J_2(e_2)=e_4$ & $\omega=a\left(e^{12}-e^{34}\right)+b\left(e^{14}-e^{23}\right), [a:b]\in\RR\mathbb{P}^1$\\
\bottomrule[1pt]
\end{tabu}}
\vskip 0.25 cm
\caption{4-dimensional Lie algebras with a symmetric complex structure}\label{table:Complex-symplectic structures}
\end{center}
\end{table}

\begin{proof}
We know that such a Lie algebra must be solvable. Proposition 5.4 in \cite{BFLM2018} takes care of the nilpotent Lie algebras, accounting for the first two lines in Table \ref{table:Complex-symplectic structures}. Hence, we concentrate on solvable non-nilpotent Lie algebras.

The proof consists of a case-by-case check; for each Lie algebra admitting both complex and symplectic structures, given in Table~\ref{tabla1}, we check the symmetry of the complex structure $J$ with respect to the generic symplectic structure $\omega$. As a first step, we check that Table~\ref{table:Complex-symplectic structures} contains complex symplectic Lie algebras.

\textbf{Case $\mathfrak{r}'_{2}$} This Lie algebra admits two non-equivalent complex structures. One sees that the first one is not symmetric with respect to any symplectic form. The second one, $J_\xi$, depends on $\xi\in\CC$ with $\Im(\xi)\neq0$ and is given by \eqref{Jxi}. The generic symplectic form is
\[
\omega=a_{12}e^{12}+a_{13}\left(e^{13}-e^{24}\right)+a_{14}\left(e^{14}+e^{23}\right)\,, \quad a_{13}^2+a_{14}^2\neq 0\,.
\]
The symmetry condition $\omega(J_\xi X,Y)=\omega(X,J_\xi Y)$ yields $a_{12}=0$ and
\[
\left\{
\begin{array}{lcc}
a_{14}(|\xi|^2-\Im(\xi))+a_{13}\Re(\xi) & = & 0\\
a_{13}(|\xi|^2-\Im(\xi))-a_{14}\Re(\xi) & = & 0\\
a_{13}(\Im(\xi)-1)-a_{14}\Re(\xi) & = & 0\\
a_{14}(\Im(\xi)-1)+a_{13}\Re(\xi) & = & 0
\end{array}
\right.
\]
For the last two equations, wee see that since $a_{13}^2+a_{14}^2\neq 0$, the only solution is $\Re(\xi)=0$, $\Im(\xi)=1$, hence $\xi=i$ and we obtain the complex symplectic structure $(J_i,\omega)$, with
\[
\omega=a_{13}\left(e^{13}-e^{24}\right)+a_{14}\left(e^{14}+e^{23}\right)\,, \quad a_{13}^2+a_{14}^2\neq 0\,.
\]
For $r\neq 0$ and $\theta\in\RR$, the Lie algebra automorphism $\psi=\psi_{r,\theta}$ given by
\[
\psi(e_1)=e_1\,, \ \psi(e_2)=e_2\,, \ \psi(e_3)=r(\cos\theta e_3-\sin\theta e_4) \ \mathrm{and} \ \psi(e_4)=r(\cos\theta e_4+\sin\theta e_3)
\]
satisfies $J_i\circ\psi=\psi\circ J_i$ and 
\[
\psi^*\omega=r(a_{13}\cos\theta-a_{14}\sin\theta)\left(e^{13}-e^{24}\right)+r(a_{13}\sin\theta+a_{14}\cos\theta)\left(e^{14}+e^{23}\right)\,.
\]
Since $a_{13}^2+a_{14}^2\neq 0$, one always finds $\theta$ such that $a_{13}\sin\theta+a_{14}\cos\theta=0$; also, by suitably choosing $r$, one normalizes the coefficient of $e^{13}-e^{24}$. Hence every complex symplectic structure on $\mathfrak{r}'_2$ is equivalent to $\left(J_i,e^{13}-e^{24}\right)$.

\textbf{Case $\mathfrak{r}_{4,-1,-1}$} 
A generic closed and non-degenerate 2-form on $\mathfrak{r}_{4,-1,-1}$ is given by
\[
\omega=a_{12}e^{12}+a_{13}e^{13}+a_{14}e^{14}+a_{24}e^{24}
+a_{34}e^{34}\,, \quad a_{12}a_{34}-a_{13}a_{24}\neq 0\,.
\]
The only complex structure $J\colon\mathfrak{r}_{4,-1,-1}\to\mathfrak{r}_{4,-1,-1}$ is given by $J(e_1)=-e_4$, $J(e_2)=e_3$. Imposing the symmetry condition we obtain $a_{14}=0$, $a_{24}=a_{13}$ and $a_{34}=-a_{12}$, thus the complex symplectic structure is
\[
\omega=a_{12}\left(e^{12}-e^{34}\right)+a_{13}\left(e^{13}+e^{24}\right)\,, \quad a_{12}^2+a_{13}^2\neq 0\,.
\]
For $r\neq 0$ and $\theta\in\RR$, the Lie algebra automorphism $\varphi=\varphi_{r,\theta}$ given by
\[
\varphi(e_1)=e_1\,, \ \varphi(e_2)=r(\cos\theta e_2-\sin\theta e_3)\,, \ \varphi(e_3)=r(\cos\theta e_3+\sin\theta e_2) \ \mathrm{and} \ \varphi(e_4)=e_4
\]
satisfies $J\circ\varphi=\varphi\circ J$ and 
\[
\varphi^*\omega=r(a_{12}\cos\theta-a_{13}\sin\theta)\left(e^{12}-e^{34}\right)+r(a_{12}\sin\theta+a_{13}\cos\theta)\left(e^{13}+e^{24}\right)\,.
\]
Since $a_{12}^2+a_{13}^2\neq 0$, one always finds $\theta$ such that $a_{12}\sin\theta+a_{13}\cos\theta=0$; moreover, by suitably choosing $r$, the coefficient of $e^{12}-e^{34}$ can be normalized. Hence every complex symplectic structure on $\mathfrak{r}_{4,-1,-1}$ is equivalent to $\left(J,e^{12}-e^{34}\right)$.
		
\textbf{Case $\mathfrak{d}_{4,2}$} This Lie algebra admits two non-equivalent complex structures. A computation shows that the first one  is not symmetric with respect to any symplectic form. The second one is given by $J_2(e_1)=e_3$ and $J_2(e_2)=e_4$. The generic symplectic form on $\frd_{4,2}$ is
\[
\omega=a_{12}\left(e^{12}-e^{34}\right)+a_{14}e^{14}+a_{23}e^{23}+a_{24}e^{24}\,, \quad a_{12}^2-a_{14}a_{23}\neq 0\,.
\]
Imposing the symmetry condition gives $a_{24}=0$ and $a_{23}=-a_{14}$, hence we get the complex symplectic structure
\[
\omega=a_{12}\left(e^{12}-e^{34}\right)+a_{14}\left(e^{14}-e^{23}\right)\,, \quad a_{12}^2+a_{14}^2\neq 0\,.
\]
For every $r\neq 0$ the map $\phi_r\colon\mathfrak{d}_{4,2}\to\mathfrak{d}_{4,2}$ given by
\[
\phi_r(e_1)=r^{-1}e_1\,, \ \phi_r(e_2)=e_2\,, \ \phi_r(e_3)=r^{-1}e_3 \ \mathrm{and} \ \phi_r(e_4)=e_4
\] 
is a Lie algebra automorphism and, moreover, $J_2\circ\phi_r=\phi_r\circ J_2$. If $a_{12}\neq 0$, then $\phi_{a_{12}}^*\omega=e^{12}-e^{34}+b\left(e^{14}-e^{23}\right)$, $b\in\RR$. One sees that, for $b\neq b'$, two such symplectic forms are not symplectomorphic. If $a_{12}=0$, then $\phi_{a_{14}}^*\omega=e^{14}-e^{23}$. Arguing with $a_{14}$ instead of $a_{12}$, we see that there is an $\RR\mathbb{P}^1$ of symplectic forms,
\[
\omega_{[a:b]}=a\left(e^{12}-e^{34}\right)+b\left(e^{14}-e^{23}\right)\,, \quad [a:b]\in\RR\mathbb{P}^1\,,
\]
which are pairwise non symplectomorphic.

To finish the proof we need to check that the rest of Lie algebras admitting both symplectic and complex structures do not admit complex symplectic structures. This is done by imposing the symmetry condition $\omega(JX,Y)=\omega(X,JY)$ to a complex structure with respect to a generic closed 2-form $\omega$, as given by Table \ref{tabla1}, and checking that the resulting $\omega$ is degenerate, hence not symplectic. To illustrate this process, let us consider $\mathfrak{rr}_{3,0}$. This Lie algebra has a unique complex structure $J$, given by $J(e_1)=e_2$, $ J(e_3)=e_4$, while the generic closed 2-form is $\omega=a_{12}e^{12}+a_{13}e^{13}+a_{14}e^{14}+a_{34}e^{34}$. Imposing the symmetry condition yields $a_{12}=a_{13}=a_{14}=a_{34}=0$, hence $\omega$ is not symplectic.
\end{proof}
	
\begin{remark}
As a result of our classification, we see that the following Lie algebras admit both complex and symplectic structures, but the complex structure is never symmetric with respect to a symplectic structure: 
\[
\mathfrak{rr}_{3,0}\,, \ \mathfrak{rr}'_{3,0}\,, \ \mathfrak{r}_2\mathfrak{r}_2\,, \ \mathfrak{r}'_{4,0,\delta}\,, \ \mathfrak{d}_{4,\lambda}\,, \lambda\neq 2\,, \ \mathfrak{d}'_{4,\delta} \ \mathrm{and} \ \mathfrak{h}_4.
\]
\end{remark}

\subsection{Lie algebras with symmetric and skew-symmetric complex structures}

The classification of 4-dimensional Lie algebras endowed with a complex structure which is skew-symmetric with respect to a symplectic structure, i.e.~of pseudo-Kähler 4-dimensional Lie algebras, was obtained by Ovando \cite{Ovando2006PK}. According to this classification, the Lie algebras $\frh_4$ and $\mathfrak{d}_{4,\lambda}$ with $\lambda\notin\left\{1,2,\frac{1}{2}\right\}$ admit both a complex structure and a symplectic structure, but the complex structure is not skew-symmetric with respect to any symplectic structure.

Combining these result with our results from Section \ref{sec:4DSYM}, we see that the 4-dimensional Lie algebras admitting a complex structure which is symmetric with respect to a symplectic structure are a subset of those admitting a complex structure which is skew-symmetric. The precise results are contained in Table \ref{table:Structures_Lie_algebras_I}.

\begin{table}[h!]
\begin{center}
{\tabulinesep=1.2mm
\begin{tabu}{ccc|ccc}
\toprule[1.5pt]
Lie algebra & Sym (cs) & Skew (pK) & Lie algebra & Sym (cs) & Skew (pK)\\
\specialrule{1pt}{0pt}{0pt}
$\mathfrak{rh}_3$ & $\checkmark$ & $\checkmark$ & $\mathfrak{r}'_{4,0,\delta}$ & $\times$ & $\checkmark$ \\
\specialrule{1pt}{0pt}{0pt}
$\mathfrak{rr}_{3,0}$ & $\times$ & $\checkmark$ & $\mathfrak{d}_{4,1}$ & $\times$ & $\checkmark$\\
\specialrule{1pt}{0pt}{0pt}
$\mathfrak{rr}'_{3,0}$ & $\times$ & $\checkmark$ & $\mathfrak{d}_{4,2}$ &$\checkmark$&$\checkmark$\\
\specialrule{1pt}{0pt}{0pt}
$\mathfrak{r}_2\mathfrak{r}_2$ & $\times$ & $\checkmark$ & $\mathfrak{d}_{4,\frac{1}{2}}$ & $\times$ & $\checkmark$\\
\specialrule{1pt}{0pt}{0pt}
$\mathfrak{r}'_2$ & $\checkmark$&$\checkmark$ & $\mathfrak{d}'_{4,\delta}$ & $\times$ & $\checkmark$\\
\specialrule{1pt}{0pt}{0pt}
$\mathfrak{r}_{4,-1,-1}$ & $\checkmark$ & $\checkmark$\\
\bottomrule[1pt]
\end{tabu}}
\vskip 0.25 cm
\caption{4-dimensional non-abelian Lie algebras with symmetric or skew-symmetric complex structures}\label{table:Structures_Lie_algebras_I}
\end{center}
\end{table}
	
In particular, we see that for 4-dimensional Lie algebras the existence of a complex symplectic pair $(J_{cs},\omega_{cs})$ implies the existence of a pseudo-Kähler pair $(J_{pK},\omega_{pK})$. In fact, the Lie algebras $\mathfrak{rh}_3$ and $\mathfrak{r}_{4,-1,-1}$ only admit one complex structure, hence $J_{pK}=J_{cs}$ in these cases. $\mathfrak{d}_{4,2}$ admits two non-equivalent complex structures $J_1$ and $J_2$. $J_1$ is skew-symmetric with respect to a symplectic form but it is not symmetric with respect to any symplectic form; $J_2$ is both symmetric and skew-symmetric with respect to certain symplectic forms. $\mathfrak{r}'_2$ admits non-equivalent complex structures $J$ and $J_\xi$, the second one being parametrized by $\xi\in\CC$ with $\Im(\xi)\neq0$. $J_i$ is symmetric with respect to some symplectic form, but not skew-symmetric with respect to any symplectic form,  while $J$ and $J_{-i}$ are skew-symmetric with respect to certain symplectic forms, but not symmetric with respect to any symplectic form. The Lie algebra $\mathfrak{r}'_2$ is not unimodular, hence, by a result of Milnor \cite{Milnor1976}, it does not admit any compact quotient. In \cite[Theorem 1.1]{Yamada2017}, Yamada constructed an 8-dimensional (compact) nilmanifold $M$ with two complex structures $J_1$ and $J_2$ such that $J_1$ is skew-symmetric with respect to some symplectic form, but not symmetric, and viceversa for $J_2$. 

%
\section{Hypersymplectic structures}\label{section:hypersymplectic}
	
We have seen that every hypersymplectic structure on a Lie algebra $\frg$ provides a complex symplectic structure and a pseudo-K\"ahler structure on $\frg$ that share the same complex structure.  
In this section we first consider the 4-dimensional case, and then present a method to obtain hypersymplectic structures starting from a complex structure and two symplectic forms with respect to which the complex structure is symmetric and skew-symmetric.

\subsection{4-dimensional Lie algebras with a hypersymplectic structure}\label{sec:4DHS}
Hypersymplectic structures on 4-dimensional Lie algebras have been classified by Andrada; since the underlying Lie algebra is in particular symplectic, the Lie algebra is necessarily solvable.
	
\begin{theorem}{\normalfont\cite[Theorem 23]{Andrada2006}}\label{theo:hypersymplectic_classification}
Let $\frg$ be a 4-dimensional Lie algebra carrying a hypersymplectic structure. Then $\frg$ is isomorphic to 
\[
\RR^4\,, \quad \mathfrak{rh}_3\,, \quad \mathfrak{r}_{4,-1,-1} \quad \mathrm{or} \quad \mathfrak{d}_{4,2}.
\]
\end{theorem}

Andrada actually gives a classification of all possible hypersymplectic metrics on each of these Lie algebras; we refer to \cite{Andrada2006} for further details. Notice that the Lie algebra $\mathfrak{r}'_2$ admits both complex symplectic and pseudo-Kähler structures (see Table~\ref{table:Structures_Lie_algebras_I}), but no hypersymplectic metric. This is perhaps not surprising in view of what we observed above: $\mathfrak{r}'_2$ has no complex structure which is at the same time symmetric and skew-symmetric with respect to certain symplectic forms.
	
The next results subsumes what we have found so far.
	
\begin{theorem}\label{theo:4-dim-CS-HS}
Let $(\frg,J)$ be a 4-dimensional solvable Lie algebra with a complex structure. Then $(\frg,J)$ admits hypersymplectic structure if and only if $(\frg,J)$ admits a pseudo-Kähler and a complex symplectic structure.
\end{theorem}

In other words, if we fix a complex structure $J$ on a 4-dimensional solvable Lie algebra~$\frg$, then~$J$ comes from a hypersymplectic structure if and only if~$J$ is at the same time symmetric and skew-symmetric with respect to certain symplectic forms.

It is natural to ask whether this result holds in higher dimension. For instance, Guan proved that if $\frg$ is a complex 4-dimensional solvable Lie algebra, then the existence of a symplectic form with respect to which the natural complex structure is skew-symmetric is necessary and sufficient for the existence of a hypersymplectic metric \cite{Guan11}. We will come back to this topic in Section \ref{subsec:examples}.
	
\subsection{Constructing hypersymplectic Lie algebras}

As we noticed in Remark \ref{rem:recursion}, hypersymplectic structures can be described as triples of symplectic forms whose recursion operators are of a prescribed type, see \cite[Section 4]{Bande-Kotschick}. In order to obtain explicit examples of hypersymplectic structures, and sticking with the interplay between symmetric and skew-symmetric complex structures, we propose here a different, although related, approach to hypersymplectic structures. This approach emphasizes the role played by the complex structure $J$, to which we attach a pair of symplectic structures $(\omega_{pK},\omega_{cs})$ such that $J$ is skew-symmetric with respect to $\omega_{pK}$ and symmetric with respect to $\omega_{cs}$.
	
Let $\frg$ be a Lie algebra endowed with a complex structure $J$ and two symplectic forms $\omega_{cs}$ and $\omega_{pK}$ with respect to which $J$ is, respectively, symmetric and skew-symmetric. Mindful of Lemma \ref{prop:E_hypersymplectic} and of \eqref{eq:bilinear_forms}, we define $E\colon\frg\to\frg$ and a scalar product $g$ on $\frg$ by 
\[
E=\omega_{pK}^{-1}\circ\omega_{cs} \quad \mathrm{and} \quad g(X,Y)=\omega_{pK}(X,JY)\,.
\]
Clearly, the first equation is equivalent to $\omega_{pK}(EX)=\omega_{cs}(X)$. In order for $(J,E,g)$ to be a hypersymplectic structure on $\frg$ we need to check that $(J,E)$ is a complex product structure and that $g$ is compatible with $(J,E)$. 

The first observation is that, independently from the fact that~$E^2$ may or not equal~$\mathrm{Id}$, by its own definition~$E$ is never~$\pm\mathrm{Id}$. Indeed, this would imply $\omega_{cs}=\pm\omega_{pK}$, which is impossible, since~$J$ is symmetric with respect to~$\omega_{cs}$ and skew-symmetric with respect to~$\omega_{pK}$.

We next show that $E$ always anti-commutes with $J$. 

\begin{lemma}\label{lem:skewEJ}
If $E\colon\frg\to\frg$ is defined by $E=\omega_{pK}^{-1}\circ\omega_{cs}$, then $E\circ J=-J\circ E$.
\end{lemma}
\begin{proof}
Since $J$ is symmetric with respect to $\omega_{cs}$, we first notice that, for $X,Y\in\frg$,
\[
\omega_{cs}(JX)(Y)=\omega_{cs}(JX,Y)=\omega_{cs}(X,JY)=\omega_{cs}(X)(JY)=(J^*\circ\omega_{cs}(X))(Y)\,,
\]
hence $\omega_{cs}\circ J=J^*\circ\omega_{cs}$, where $J^*\colon\frg^*\to\frg^*$ is the dual map to $J$. Similarly, $\omega_{pK}\circ J=-J^*\circ\omega_{pK}$, and $J\circ\omega_{pK}^{-1}=-\omega_{pK}^{-1}\circ J^*$. Therefore
\[
E\circ J=\omega_{pK}^{-1}\circ\omega_{cs}\circ J=\omega_{pK}^{-1}\circ J^*\circ\omega_{cs}=-J\circ \omega_{pK}^{-1}\circ\omega_{cs}=-J\circ E\,.
\]
\end{proof}

Now, we prove that $E$ being an almost product structure is equivalent to it being compatible with $g$, that is, to the second relation in \eqref{eq:compatible_metric}.

\begin{lemma}\label{lemma:gE_compatible}
If $E=\omega_{pK}^{-1}\circ\omega_{cs}$, then
\begin{enumerate}[label=\arabic*.]
\item $g(EX,Y)=-g(X,EY)$;
\item $g(EX,EY)=-g(X,Y)$ if and only if $E$ is an almost product structure.
\end{enumerate}
\end{lemma}

\begin{proof}
For the first point, we have, for $X,Y\in\frg$,
\begin{align*}
g(EX,Y) &=\omega_{pK}(EX,JY)=\omega_{pK}\left(\left(\omega_{pK}^{-1}\circ\omega_{cs}\right)(X)\right)(JY)=\omega_{cs}(X)(JY)\\
&=\omega_{cs}(X,JY)
\end{align*}
and
\begin{align*}
g(X,EY) &=\omega_{pK}(X,JEY)=\omega_{pK}(EJY,X)=\omega_{pK}\left(\left(\omega_{pK}^{-1}\circ\omega_{cs}\right)(JY)\right)(X)\\
&=\omega_{cs}(JY)(X)=\omega_{cs}(JY,X)=-\omega_{cs}(X,JY)\,,
\end{align*}
where we use Lemma \ref{lem:skewEJ} and the fact that $\omega_{pK}\in\Lambda^2\frg^*$ in the second equality. Therefore, we see that, for every $X,Y\in\frg$,
\[
g(EX,EY)=-g(X,E^2Y)\,,
\]
and the latter equals $-g(X,Y)$ if and only if $E^2=\mathrm{Id}$.
\end{proof}

\begin{lemma}\label{lemma:E_integrable}
Suppose $E=\omega_{pK}^{-1}\circ\omega_{cs}$ is an almost product structure. Then $E$ is a product structure.
\end{lemma}
\begin{proof}
This is basically a reproduction of Lemma 2 in \cite{Bande-Kotschick}. Since $E$ is an almost product structure, then $\frg=\frg_+\oplus\frg_-$. Consider the 2-forms $\omega_{\pm}=\omega_{pK}\pm\omega_{cs}$. We claim that
\[
\frg_{\pm}=\ker(\omega_{\mp})\coloneqq\{X\in\frg \mid \omega_{\mp}(X)=0\}\,.
\]
Indeed,
\[
\omega_\mp(X)=(\omega_{pK}\mp\omega_{cs})(X)=\omega_{pK}(X)\mp\omega_{pK}(EX)=\omega_{pK}(X\mp EX)\,;
\]
since $\omega_{pK}$ is non-degenerate, the latter is zero if and only if $X\in\frg_{\pm}$. Now, $E$ is integrable if and only if $\frg_{\pm}$ are subalgebras of $\frg$. Since $\omega_\mp$ is closed, given $X,Y\in\frg_\pm$, one gets $[X,Y]\in\frg_\pm$, hence $E$ is integrable. 
\end{proof}

We are now ready to state our construction of a hypersymplectic structure on a Lie algebra $\frg$ starting with a complex structure $J$ and two symplectic forms $\omega_{pK}$ and $\omega_{cs}$ with respect to which $J$ is, respectively, skew-symmetric and symmetric.
	
\begin{theorem}\label{theo:main}
Let $(\frg,J)$ be a Lie algebra endowed with a complex structure and let $\omega_{pK}$ and $\omega_{cs}$ be symplectic forms with respect to which $J$ is, respectively, skew-symmetric and symmetric. Define $E\colon\frg\to\frg$ by $E=\omega_{pK}^{-1}\circ\omega_{cs}$ and let $g=-\omega_{pK}\circ J$ be the pseudo-Kähler metric. Then $(J,E,g)$ is a hypersymplectic structure if and only if $E$ is an almost product structure.
\end{theorem}

\begin{proof}
One direction is clear. For the other one, since $E$ is an almost product structure, then it is a product structure by Lemma \ref{lemma:E_integrable}. By Lemma \ref{lem:skewEJ} $(J,E)$ is a complex product structure. By Lemma \ref{lemma:gE_compatible}, the metric $g$ is compatible with $(J,E)$. Since $\omega_{pK}$ and $\omega_{cs}$ are closed by assumption, the metric $g$ is hypersymplectic and $(J,E,g)$ is a hypersymplectic structure.
\end{proof}

\begin{corollary}
If $E\colon\frg\to\frg$ defined above is an almost product structure, then $\omega_{cs}$ is parallel with respect to the Levi-Civita connection of the pseudo-Kähler metric.
\end{corollary}
\begin{proof}
By Theorem \ref{theo:main} the metric $g$ is hypersymplectic, hence its holonomy reduces to $\mathrm{Sp}(2n,\RR)$, where $4n=\dim\frg$. Then $E$ is parallel with respect to the Levi-Civita connection of $g$, and so is $\omega_{pK}$, hence the same holds for $\omega_{cs}$.
\end{proof}

\begin{remark}
The above theorem can be rephrased in the language of $G$-structures, for $G=\mathrm{Sp}(2n,\RR)$: as soon as the triple $(J,\omega_{pK},\omega_{cs})$ defines an $\mathrm{Sp}(2n,\RR)$-structure, i.e.~as soon as $E$ is an almost product structure, the structure is automatically torsion-free.
\end{remark}

\begin{remark}
Using this result and the classification of complex symplectic and pseudo-Kähler structure on 4-dimensional Lie algebras, one could obtain a different proof of Theorem \ref{theo:hypersymplectic_classification}.
\end{remark}

\subsection{Examples}\label{subsec:examples} We apply Theorem \ref{theo:main} to study the interplay between two symplectic structures with respect to which a fixed complex structure is, respectively, symmetric and skew-symmetric.

\begin{example}\label{ex:1}
Consider the 6-dimensional nilpotent Lie algebra $\frh_4=(0^4,12,14+23)$ and set $\frg=\frh_4\oplus\RR^2$. The following complex structure on $\frg$ has been adapted from \cite[Proposition 5.2.12]{Latorre2016}:
\[
J(e_1)=e_2\,, \quad J(e_3)=-e_4\,, \quad J(e_5)=2e_7 \quad \mathrm{and} \quad J(e_6)=-e_8\,.
\]
Then $J$ is symmetric with respect to
\begin{align*}
\omega_{cs}&=a_{13}\left(e^{13}+e^{24}\right)+a_{14}\left(e^{14}-e^{23}\right)+a_{15}\left(2e^{15}-e^{27}\right)+a_{17}\left(e^{17}+2e^{25}\right)\\
		&+a_{16}\left(2e^{16}+2e^{28}-2e^{35}-e^{47}\right)+a_{18}\left(2e^{18}-2e^{26}+e^{37}-2e^{45}\right)\,,
\end{align*}
subject to $a_{16}^2+a_{18}^2\neq 0$, and skew-symmetric with respect to
\begin{align*}
\omega_{pK}&=b_{12}e^{12}+b_{13}\left(e^{13}-e^{24}\right)+b_{14}\left(e^{14}+e^{23}\right)+b_{15}\left(2e^{15}+e^{27}\right)+b_{17}(e^{17}-2e^{25})\\
		&+b_{16}\left(2e^{16}-2e^{28}-2e^{35}+e^{47}\right)+b_{18}\left(2e^{18}+2e^{26}+e^{37}+2e^{45}\right)+b_{34}e^{34}\,,
\end{align*}
subject to $b_{16}^2+b_{18}^2\neq 0$. Let us consider the endomorphism $E\colon\frg\to\frg$ defined by $E=\omega_{pK}^{-1}\circ\omega_{cs}$. By choosing $a_{ij}=b_{ij}$ for every $i,j$, as well as $b_{14}=0$, one computes that $E^2=\mathrm{Id}$. Hence, for every such choice of the parameters, $(J,E,g)$ is a hypersymplectic structure by Theorem \ref{theo:main}.

$(J,E,g)$ gives a left-invariant hypersymplectic structure on $G$, the only connected, simply connected Lie group $G$ with Lie algebra $\frg$. By Malcev theorem \cite{Maltsev1949} $G$ admits a lattice $\Gamma$, and hence the nilmanifold $\Gamma\backslash G$ is hypersymplectic. Since $\frg$ is 2-step nilpotent and has 4-dimensional center, $(\Gamma\backslash G,J)$ is an example of a Kodaira manifold. Hypersymplectic metrics on Kodaira manifolds have been investigated in \cite{FPPS}. In fact, it can be shown that the pair $(\frg,J)$ coincides with the corresponding pair of Example 4 in \cite{FPPS}.
\end{example}

\begin{example}\label{ex:2}
We provide an example of a nilmanifold endowed with a left-invariant complex structure $J$ such that, for every possible choice of a left-invariant complex symplectic structure $\omega_{cs}$ and of a left-invariant pseudo-Kähler structure $\omega_{pK}$, $E=\omega_{pK}^{-1}\circ\omega_{cs}$ is not an almost product structure.

On the irreducible 8-dimensional, 4-step nilpotent Lie algebra $\frg=(0^3,12,13+24,14-23,15+26,16+7\cdot 25+8\cdot 34)$ we consider the complex structure given by
\[
J(e_1)=e_2\,, \quad J(e_3)=-3e_4\,, \quad J(e_5)=-e_6 \quad \mathrm{and} \quad J(e_7)=3e_8\,.
\]
The most general 2-form with respect to which $J$ is symmetric is
\begin{align*}
\omega_{cs}&=a_{13}\left(3e^{13}+e^{24}\right)+a_{14}\left(e^{14}-3e^{23}\right)+a_{15}\left(e^{15}+e^{26}\right)+a_{16}\left(e^{16}-e^{25}\right)\\
		&+a_{17}\left(3e^{17}-e^{28}-12e^{35}+4e^{46}\right)+a_{18}\left(e^{18}+3e^{27}+12e^{36}+4e^{45}\right)\,,
\end{align*}
subject to $a_{17}^2+a_{18}^2\neq 0$, and the most general 2-form with respect to which $J$ is skew-symmetric is
\begin{align*}
\omega_{pK}&=b_{12}e^{12}+b_{13}\left(3e^{13}-e^{24}\right)+b_{14}\left(e^{14}+3e^{23}\right)+b_{16}(e^{16}+e^{25}+2e^{34})\\
		&+b_{17}\left(3e^{17}+e^{28}+6e^{35}+2e^{46}\right)\,,
\end{align*}
subject to $b_{17}\neq 0$. Consider the endomorphism $E\colon\frg\to\frg$ defined by $E=\omega_{pK}^{-1}\circ\omega_{cs}$. One checks that the entries $(1,1)$ and $(3,3)$ of $E^2$ are $\frac{a_{17}^2+a_{18}^2}{b_{17}^2}$ and $\frac{4(a_{17}^2+a_{18}^2)}{b_{17}^2}$. Imposing that they are equal to 1, we obtain $a_{17}=a_{18}=0$, which contradicts the fact that the generic complex symplectic form is non degenerate. Thus $E$ is never an almost product structure, and, consequently, $(\frg,J)$ does not admit {\em any} hypersymplectic metric, by Theorem \ref{theo:main}. This shows that in dimension $4n$ with $n\geq 2$ the analogue of Theorem \ref{theo:4-dim-CS-HS} does not hold. By Malcev theorem \cite{Maltsev1949} $G$ admits a lattice $\Gamma$, and the corresponding nilmanifold $\Gamma\backslash G$ inherits a triple $(J,\omega_{cs},\omega_{pK})$ which never combines to give a left-invariant hypersymplectic structure.
\end{example}


\section{A 4-step nilmanifold with hypersymplectic metrics}\label{sec:examples}

In this section we apply our construction and obtain two 1-parameter families of left-invariant hypersymplectic structures on a compact nilmanifold whose underlying Lie algebra is 4-step nilpotent; to the best of our knowledge, this is the first example with such nilpotency step. The corresponding metrics are complete, one is flat and the other one is non-flat. 

On the irreducible, 8-dimensional 4-step nilpotent Lie algebra
\[
\frh=(0^3,12,13,14+23,15,16+2\cdot 25+34)
\]
we consider the family of complex structures $J_c$ given by
\[
J_c(e_1)=\frac{c+1}{c}e_2\,, \quad J_c(e_3)=-e_4\,, \quad J_c(e_5)=\frac{1}{c}e_6 \quad \mathrm{and} \quad J_c(e_7)=\frac{3+2c}{c}e_8\,,
\]
where $c\in(0,\infty)$ is a parameter. Recall that two complex structures $J_1,J_2$ on a Lie algebra~$\frh$ are {\em equivalent} if there exists a Lie algebra automorphism $\psi\colon\frh\to\frh$ such that $\psi\circ J_1=J_2\circ \psi$.

\begin{proposition}
Two complex structures $J_c$ and $J_{c'}$ on $\frh$ are equivalent if and only if $c=c'$.
\end{proposition}
\begin{proof}
In terms of the basis $\{e_k\}_{k=1}^8$, one can check that every Lie algebra isomorphism $\psi\colon\frh\to\frh$ is given by the matrix
\[
\psi=\begin{pmatrix}
\lambda^1_1 & 0 & 0 & 0 & 0 & 0 & 0 & 0 \\[4pt] 
\lambda^1_2 & \lambda^2_2  & 0 & 0 & 0 & 0 & 0 & 0 \\[4pt] 
\lambda^1_3 & 0  & (\lambda^1_1)^2 & 0 & 0 & 0 & 0 & 0 \\[4pt] 
\lambda^1_4 & \lambda^2_4 & \lambda^3_4 & \lambda^1_1\lambda^2_2 & 0 & 0 & 0 & 0 \\[4pt] 
\lambda^1_5 & \lambda^2_5 & \frac{\lambda^2_4(\lambda^1_1)^2}{\lambda^2_2} 
	& 0 & (\lambda^1_1)^3 & 0 & 0 & 0 \\[4pt]
\lambda^1_6 & \lambda^2_6 & \lambda^3_6 & \lambda^4_6 
	& \lambda^5_6 
	& (\lambda^1_1)^2\lambda^2_2 & 0 & 0 \\[4pt]
\lambda^1_7 & \lambda^2_7 & \lambda^3_7 & \lambda^1_1\lambda^2_5 
	& \frac{(\lambda^1_1)^3\lambda^2_4}{\lambda^2_2} & 0 & (\lambda^1_1)^4 & 0 \\[4pt]
\lambda^1_8 & \lambda^2_8 & \lambda^3_8 & \lambda^4_8 
	& \lambda^5_8  & (\lambda^1_1)^2\lambda^2_4 
	& \lambda^7_8
	& (\lambda^1_1)^3\lambda^2_2
\end{pmatrix}
\]
where $\lambda^1_1\lambda^2_2\neq 0$ and
\begin{flalign*}
\lambda^4_6 &=\lambda^1_1\lambda^2_4-\lambda^1_3\lambda^2_2\,,
	&
\lambda^4_8 &=\lambda^1_1\lambda^2_6+2\lambda^1_2\lambda^2_5
	-2\lambda^1_5\lambda^2_2+\lambda^1_3\lambda^2_4\,, \ \quad
\lambda^7_8=(\lambda^1_1)^2(3\lambda^1_1\lambda^1_2+\lambda^3_4),\\
\lambda^5_6 &=\lambda^1_1(\lambda^1_1\lambda^2_2+\lambda^3_4)\,,
	&
\lambda^5_8 &=\lambda^1_1\lambda^3_6
+\frac{2\,\lambda^1_2\lambda^2_4\,(\lambda^1_1)^2}{\lambda^2_2}
		+\lambda^1_3\lambda^3_4-\lambda^1_4(\lambda^1_1)^2.
	&
\end{flalign*}
If $J_c$ and $J_{c'}$ are equivalent, then
\begin{align}
\frac{1+c}{c}\,\psi(e_2) = \psi(J_c\,e_1) = J_{c'}(\psi e_1)
	=& -\frac{c'}{1+c'}\,\lambda^1_2\,e_1 + \frac{1+c'}{c'}\,\lambda^1_1\,e_2 + \lambda^1_4\,e_3
	- \lambda^1_3\,e_4 \label{eq1}\\
	&-c'\lambda^1_6\,e_5 +\frac{\lambda^1_5}{c'}\,e_6
	-\frac{c'}{3+2c'}\,\lambda^1_8\,e_7+\frac{3+2c'}{c'}\,\lambda^1_7\,e_8 \notag\\
-c\,\psi(e_5) = \psi(J_c\,e_6) =  J_{c'}(\psi e_6)
	=& -c'\,(\lambda^1_1)^2\,\lambda^2_2\,e_5 -\frac{c'}{3+2c'}\,(\lambda^1_1)^2\lambda^2_4\,e_7\,.
\label{eq2}
\end{align}
Equalling the coefficients of~$e_2$ in~\eqref{eq1} and those of~$e_5$ in~\eqref{eq2}, we obtain the equations:
$$\frac{1+c}{c}\,\lambda^2_2 = \frac{1+c'}{c'}\,\lambda^1_1, 
\qquad
c\,(\lambda^1_1)^3=c'\,(\lambda^1_1)^2\lambda^2_2.$$
If we solve for $\lambda^2_2$ in the second equation and replace it in the first one, we conclude that $c=c'$.
\end{proof}

The complex structure $J_c$ is symmetric with respect to
\begin{align}\label{ex:3.1.1}
\omega_{cs}&=a_{13}\big((c+1)e^{13}+ce^{24}\big)
+a_{14}\big((c+1)e^{14}-ce^{23}\big)\nonumber\\
&+a_{17}\big((c+1)(2c+3)e^{17}-c^2e^{28}+ce^{35}+c^2e^{46}\big)\nonumber\\
&+a_{18}\big((c+1)e^{18}+(2c+3)e^{27}+ce^{36}-e^{45}\big)\,,
\end{align}
subject to the condition $a_{17}^2+a_{18}^2\neq 0$. As for skew-symmetry, we have:
\begin{align}\label{ex:3.2.1}
\omega_{pK}&=b_{12}e^{12}+b_{13}\big((c+1)e^{13}-ce^{24}\big)
+b_{14}\big((c+1)e^{14}+ce^{23}\big)\nonumber\\
&+b_{16}\big((c+1)e^{16}-e^{25}-(c+2)e^{34}\big)\nonumber\\
&+b_{17}\big((c+1)(2c+3)e^{17}+c^2e^{28}+ce^{35}-c^2e^{46}\big)\nonumber\\
&+ b_{18}\big((c+1)e^{18}-(2c+3)e^{27}+(5c+6)e^{36}+(4c+5)e^{45}\big)
\end{align}
with $b_{17}^2+b_{18}^2\neq 0$ and $(2c^2-3)b_{18}=0$. 

\begin{lemma}\label{lem:ap}
The endomorphism $E\colon\frh\to\frh$, $E=\omega_{pK}^{-1}\circ\omega_{cs}$, is an almost product structure if and only if
\begin{itemize}
\item $b_{18}=0$;
\item $c^2 a_{17}^2+a_{18}^2=c^2 b_{17}^2$;
\item $b_{13}=-\frac{1}{2\,c^2\,b_{17}}\left( (2+c)\,b_{16}^2+2\,c\,a_{14}\,a_{18}-2\,c^2\,a_{13}\,a_{17} \right)$.
\end{itemize}
\end{lemma}
\begin{proof}
One checks that the entries $(1,1)$ and $(3,3)$ of $E^2$ are, respectively, 
$$\frac{c^2 a_{17}^2+a_{18}^2}{c^2 b_{17}^2+b_{18}^2} \text{\quad and \quad}
\frac{c^2a_{17}^2+a_{18}^2)}{c^2b_{17}^2+(5+4c)^2 b_{18}^2}.$$
Since they must be equal to 1, one concludes that $b_{18}=0$ (since $c>0$)
and $c^2 a_{17}^2+a_{18}^2=c^2 b_{17}^2$, with $b_{17}\neq 0$. Consequently, $E^2$ becomes
\[
E^2=\begin{pmatrix}
\mathrm{Id} & 0\\ D & \mathrm{Id}
\end{pmatrix}
\]
where $D=\text{diag}\big((1+c)\,\alpha,\, \alpha,\, -\frac{c}{3+2c}\,\alpha,\, \alpha\big)$
with 
\begin{equation*}
\begin{split}
\alpha&=\frac{1}{c^5\,b_{17}^4}\big( 
	(c^2 a_{17}^2+a_{18}^2)((2+c)b_{16}^2+2c^2b_{13}b_{17}) +
	2c^3b_{17}^2\,(a_{14}a_{18}-c\,a_{13}a_{17})
	\big) \\
&= \frac{1}{c^3\,(c^2a_{17}^2+a_{18}^2)}\big(
	(2+c)\,b_{16}^2 + 2c^2b_{13}b_{17} +2ca_{14}a_{18}-2c^2a_{13}a_{17}
	\big)\,.
\end{split}
\end{equation*}
Now $D=0$ if and only if $\alpha=0$. Solving for $b_{13}$ from $\alpha=0$ one gets
\[
b_{13}=-\frac{1}{2\,c^2\,b_{17}}\left((2+c)\,b_{16}^2+2\,c\,a_{14}\,a_{18}-2\,c^2\,a_{13}\,a_{17}\right).
\]
\end{proof}

We first consider the complex symplectic form
\[
\widehat{\omega}_{cs}=(c+1)(2c+3)e^{17}-c^2e^{28}+ce^{35}+c^2e^{46}
\]
and the pseudo-Kähler form
\[
\widehat{\omega}_{pK}=(c+1)(2c+3)e^{17}+c^2e^{28}+ce^{35}-c^2e^{46}\,,
\]
obtained by setting all the parameters to zero in \eqref{ex:3.1.1} and \eqref{ex:3.2.1}, except for $a_{17}=b_{17}=1$. If $\widehat{E}=\widehat{\omega}_{pK}^{-1}\circ\widehat{\omega}_{cs}$ and $\widehat{g}=-\widehat{\omega}_{pK}\circ J$, then $\widehat{E}^2=\mathrm{Id}$ by Lemma \ref{lem:ap}. Hence, by Theorem \ref{theo:main}, $(J_c,\widehat{E},\widehat{g})$ is a family of hypersymplectic structures on $\frh$. The Levi-Civita connection~$\widehat{\nabla}$ of~$\widehat{g}$ is given by
\begin{align}\label{LC1}
\widehat{\nabla}_{e_1}e_1&=\frac{c+1}{c}e_3\,, & \widehat{\nabla}_{e_1}e_2&=-e_4\,, & \widehat{\nabla}_{e_1}e_3&=ce_5\,, & \widehat{\nabla}_{e_1}e_4&=-e_6\,,\nonumber\\ 
\widehat{\nabla}_{e_1}e_5&=-\frac{1}{2c+3}e_7\,, & \widehat{\nabla}_{e_1}e_6&=-e_8\,, & \widehat{\nabla}_{e_3}e_1&=(c+1)e_5\,, & \widehat{\nabla}_{e_3}e_2&=e_6\,,\nonumber\\
\widehat{\nabla}_{e_3}e_3&=\frac{c}{2c+3}e_7\,, & \widehat{\nabla}_{e_3}e_4&=-e_8\,, & \widehat{\nabla}_{e_5}e_1&=\frac{2c+2}{2c+3}e_7 & \widehat{\nabla}_{e_5}e_2&=2e_8\,.
\end{align}
A computation shows that the curvature vanishes, hence $(J_c,\widehat{E},\widehat{g})$ is a family of flat hypersymplectic structures on~$\frh$. 

\begin{lemma}
$\widehat{\nabla}$ is complete.
\end{lemma}
\begin{proof}
Since $\widehat{\nabla}$ is flat, its completeness is equivalent to $\rho\colon\frh\to\mathfrak{gl}(\frh)$, defined by $\rho(x)(y)=\widehat{\nabla}_yx$, being a nilpotent endomorphism for every $x,y\in\frh$ (see \cite[Theorem 1]{Segal}). This is the case, as one sees by looking at \eqref{LC1}. 
\end{proof}
We consider again the complex symplectic form
\[
\omega_{cs}=(c+1)(2c+3)e^{17}-c^2e^{28}+ce^{35}+c^2e^{46}\,,
\]
this time with the pseudo-Kähler form obtained by setting $b_{12}=b_{14}=b_{18}=0$, $b_{16}=b_{17}=1$ and $b_{13}=-\frac{c+2}{2c^2}$ in \eqref{ex:3.2.1}, namely,
\begin{align*}
\omega_{pK}&=-\frac{(c+1)(c+2)}{2c^2}e^{13}+(c+1)e^{16}+(c+1)(2c+3)e^{17}\\
&+\frac{c+2}{2c}e^{24}-e^{25}+c^2e^{28}
-(c+2)e^{34}+ce^{35}-c^2e^{46}\,.
\end{align*}
We set $E=\omega_{pK}^{-1}\circ \omega_{cs}$ and $g=-\omega_{pK}\circ J$; again, by Lemma \ref{lem:ap} one has $E^2=\mathrm{Id}$ and by Theorem~\ref{theo:main}, $(J_c,E,g)$ is a family of hypersymplectic structures on $\frh$. The Levi-Civita connection $\nabla$ of $g$ is given by
\begin{align*}
\nabla_{e_1}e_1&=\frac{c+1}{c}e_3+\frac{(c+1)^2}{c^3}e_6-\frac{(c+1)^2}{c^3(2c+3)}e_7\,, & 
\nabla_{e_2}e_4&=\frac{1}{c}e_8\,,\\
\nabla_{e_1}e_2&=-e_4-\frac{c+1}{c}e_5-\frac{c+1}{c^3}e_8\,, & 
\nabla_{e_3}e_1&=(c+1)e_5+\frac{c+1}{c^2}e_8\,,\\ 
\nabla_{e_1}e_3&=ce_5+\frac{c+1}{c^2}e_8\,, & 
\nabla_{e_3}e_2&=e_6-\frac{1}{2c+3}e_7\,,\\
\nabla_{e_1}e_4&=-e_6+\frac{c+1}{c(2c+3)}e_7\,, &
\nabla_{e_3}e_3&=\frac{c}{2c+3}e_7\,,\\
\nabla_{e_1}e_5&=-\frac{1}{2c+3}e_7\,, & 
\nabla_{e_3}e_4&=-e_8\,,
\end{align*}
\begin{align*}
\nabla_{e_1}e_6&=-e_8\,, & 
\nabla_{e_4}e_1&=\frac{c+1}{c(2c+3)}e_7\,,\\
\nabla_{e_2}e_1&=-\frac{c+1}{c}e_5-\frac{c+1}{c^3}e_8\,, &
\nabla_{e_4}e_2&=\frac{1}{c}e_8\,,\\
\nabla_{e_2}e_2&=-\frac{1}{c}e_6+\frac{1}{c(2c+3)}e_7\,, &
\nabla_{e_5}e_1&=\frac{2c+2}{2c+3}e_7\,,\\
\nabla_{e_2}e_3&=-\frac{1}{2c+3}e_7\,,&
\nabla_{e_5}e_2&=2e_8\,.
\end{align*}
Then, we see that the curvature of $g$ is non-zero; indeed,  $R(e_1,e_2)(e_1)=\frac{3(c+1)}{c(2c+3)}e_7$ and $R(e_1,e_2)(e_2)=\frac{3}{c}e_8$. Hence $(J_c,E,g)$ is a family of non-flat hypersymplectic structures on~$\frh$.
\begin{lemma}
$\nabla$ is complete.
\end{lemma}
\begin{proof}According to \cite{Guediri}, $\nabla$ is complete if and only if each curve $x\colon I\to\frh$, solution of $x'=-\nabla_xx$, is defined for every $t\in\RR$. If $x(t)=\sum_{i=1}^8x_i(t)e_i$, we obtain the following system of ODE's:
\[
\left\{\begin{array}{ccl}
x_1' & = & 0\\
x_2' & = & 0\\
x_3' & = & -\frac{c+1}{c}x_1^2\\
x_4' & = & x_1x_2\\
x_5' & = & \frac{2(c+1)}{c}x_1x_2-(2c+1)x_1x_3\\
x_6' & = & -\frac{(c+1)^2}{c^3}x_1^2+\frac{1}{c}x_2^2+x_1x_4-x_2x_3\\[2pt]
x_7' & = & \frac{(c+1)^2x_1^2-c^2x_2^2-c^4x_3^2}{c^3(2c+3)}-\frac{2(c+1)}{c(2c+3)}x_1x_4-\frac{2c+1}{2c+3}x_1x_5+\frac{2}{2c+3}x_2x_3\\
x_8' & = & \frac{2(c+1)}{c^3}x_1x_2-\frac{2(c+1)}{c^2}x_1x_3-\frac{2}{c}x_2x_4+x_1x_6+x_3x_4-2x_2x_5\\
\end{array}
\right.
\]
A tedious but straightforward computation shows that $x_i(t) $ is a polynomial in $t$, for each $i=1,\ldots,8$, hence $x(t)$ is defined for all $t\in\RR$, and $\nabla$ is complete.
\end{proof}

Since the Lie algebra $\frh$ has a rational structure, the only connected, simply connected nilpotent Lie group $H$ with Lie algebra $\frh$ admits a lattice $\Delta$ by Malcev theorem \cite{Maltsev1949}, and the nilmanifold $\Delta\backslash H$ admits two 1-parameter families of hypersymplectic structures, one flat and the other non-flat, albeit, clearly, Ricci-flat.

Putting all these results together, we obtain the following theorem:

\begin{theorem}\label{theo:example}
For every $n\geq 2$ there exists a $4n$-dimensional nilmanifold $\Gamma\backslash G$, with $\frg=\mathrm{Lie}(G)$ 4-step nilpotent, which admits two families of non-equivalent hypersymplectic structures $(J_c,\widehat{E},\widehat{g})$ and $(J_c,E,g)$, indexed by a parameter $c>0$, such that $\widehat{g}$ is flat and complete and $g$ is non-flat and complete.
\end{theorem}
\begin{proof}
The above discussion deals with the case $n=2$. To obtain the examples in dimension $4n$, $n\geq 3$, it suffices to choose $\frg=\frh\oplus\RR^{4(n-2)}$; for a flat example we consider the product of the hypersymplectic structure $(J_c,\widehat{E},\widehat{g})$ with the flat hypersymplectic structure on $\RR^{4(n-2)}$. To obtain a non-flat example, one considers the product of the hypersymplectic structure $(J_c,E,g)$ with the flat hypersymplectic structure on $\RR^{4(n-2)}$
\end{proof}

\begin{remark}
The examples of hypersymplectic metrics constructed in \cite{AndradaDotti} are at most 3-step nilpotent. It was proved in \cite[Corollary 5.15]{BFLM2018} that the nilpotency step of an 8-dimensional complex symplectic nilpotent Lie algebra is at most 4. Hence Theorem \ref{theo:example} is optimal in dimension 8, meaning that one cannot reach a higher nilpotency step. 
\end{remark}

\bibliographystyle{plain}
\bibliography{bibliography}

\end{document}